\documentclass[a4paper,11pt]{article}
\usepackage{amsmath,amssymb,enumerate,color}
\usepackage{amsthm,color}
\newcommand{\R}{\mathbb{R}}

\newcommand{\N}{\mathbb{N}}
\newcommand{\ep}{\varepsilon}
\newcommand{\pa}{\partial}

\newcommand{\lr}[1]{{}\langle{}#1{}\rangle{}}

\newtheorem{theorem}{Theorem}[section]
\newtheorem{lemma}[theorem]{Lemma}
\newtheorem{proposition}[theorem]{Proposition}
\newtheorem{corollary}[theorem]{Corollary}
\theoremstyle{remark}
\newtheorem{remark}{Remark}[section]
\theoremstyle{definition}

\newtheorem{definition}{Definition}[section]

\numberwithin{equation}{section}

\makeatletter
\def\@cite#1#2{[{{\bfseries #1}\if@tempswa , #2\fi}]}
\makeatother                  
\begin{document}
%\linenumbers
\begin{center}
\Large{{\bf
Global existence of solutions to semilinear damped\\ wave equation 
with slowly decaying \\inital data in exterior domain
}}
\end{center}

\vspace{5pt}

\begin{center}
%Masahiro Ikeda%
%\footnote{
%Department of Mathematics, Faculty of Science and Technology, Keio University, 3-14-1 Hiyoshi, Kohoku-ku, Yokohama, 223-8522, Japan/Center for Advanced Intelligence Project, RIKEN, Japan, 
%E-mail:\ {\tt masahiro.ikeda@keio.jp/masahiro.ikeda@riken.jp}}
%and
Motohiro Sobajima%
\footnote{
Department of Mathematics, 
Faculty of Science and Technology, Tokyo University of Science,  
2641 Yamazaki, Noda-shi, Chiba, 278-8510, Japan,  
E-mail:\ {\tt msobajima1984@gmail.com}}
\end{center}

%%%%     ABSTRACT     %%%%%%
\newenvironment{summary}{\vspace{.5\baselineskip}\begin{list}{}{%
     \setlength{\baselineskip}{0.85\baselineskip}
     \setlength{\topsep}{0pt}
     \setlength{\leftmargin}{12mm}
     \setlength{\rightmargin}{12mm}
     \setlength{\listparindent}{0mm}
     \setlength{\itemindent}{\listparindent}
     \setlength{\parsep}{0pt}
     \item\relax}}{\end{list}\vspace{.5\baselineskip}}
\begin{summary}
{\footnotesize {\bf Abstract.}
In this paper, we discuss the global existence of weak solutions to 
the semilinear damped wave equation 
\begin{equation*}%\label{eq:abst}
\begin{cases}
\pa_t^2u-\Delta u + \pa_tu = f(u)
& \text{in}\ \Omega\times (0,T),
\\
u=0
& \text{on}\ \pa\Omega\times (0,T),
\\
u(0)=u_0, \pa_tu(0)=u_1
& \text{in}\  \Omega,
\end{cases}
\end{equation*}
in an exterior domain $\Omega$ in $\R^N$ $(N\geq 2)$, 
where $f:\R\to \R$ is a smooth function 
behaves like $f(u)\sim |u|^p$.
From the view point of weighted energy estimates 
given by Sobajima--Wakasugi \cite{SoWa4}, 
the existence of global-in-time solutions 
with small initial data in the sense of 
$\lr{x}^{\lambda}u_0, \lr{x}^{\lambda}\nabla u_0, \lr{x}^{\lambda}u_1\in L^2(\Omega)$
with $\lambda\in (0,\frac{N}{2})$
is shown 
under the condition $p\geq 1+\frac{4}{N+2\lambda}$. 
The sharp lower bound for the lifespan of blowup solutions
with small initial data $(\ep u_0,\ep u_1)$ is also given. 

}
\end{summary}

{\footnotesize{\it Mathematics Subject Classification}\/ (2010): %
Primary: 
%35K58, %	Semilinear parabolic equations
%35L70 % Nonlinear second-order hyperbolic equations
%35L71, %	Semilinear second-order hyperbolic equations
%35Q55, %	NLS-like equations (nonlinear Schr\"odinger)
%35Q56. % Ginzburg-Landau equations
35L20,  %	Initial-boundary value problems for second-order hyperbolic equations
35B33   % 	Critical exponents
}


%%%%%%%%%%%%%%%%%%%%%%%%%%%%%%%%%%%%%%%%%%%%%%%%%%%%%%%%%%%
%%%%%%                                               %%%%%%
%                                                         %
%           Section 1 : Introduction and result           %
%                                                         %
%%%%%%                                               %%%%%%
%%%%%%%%%%%%%%%%%%%%%%%%%%%%%%%%%%%%%%%%%%%%%%%%%%%%%%%%%%%
%\tableofcontents
%\newpage
\section{Introduction}
In this paper, we consider 
the initial-boundary value problem of the semilinear damped wave equation 
\begin{equation}\label{eq:ndw}
\begin{cases}
\pa_t^2u(x,t)-\Delta u(x,t) + \pa_tu(x,t) = f\big(u(x,t)\big), 
& (x,t) \in \Omega\times (0,T),
\\
u(x,t)=0,
& (x,t) \in \pa\Omega\times (0,T),
\\
u(x,0)=u_0(x), 
& x \in \Omega,
\\
\pa_tu(x,0)=u_1(x),
& x \in \Omega,
\end{cases}
\end{equation}
where $\pa_t=\frac{\pa}{\pa t}$, $\Delta=\sum_{j=1}^N\frac{\pa^2}{\pa x_j^2}$
and $\Omega\subset\R^N$ $(N\in \N, N\geq 2)$ 
is an exterior domain (that is, $\R^N\setminus \Omega$ is bounded) 
having a smooth boundary $\pa\Omega$.  
The function $u:\overline{\Omega}\times [0,T)\to \R$ is unknown, 
the pair $(u_0,u_1)\in H^1_0(\Omega)\times L^2(\Omega)$ 
is given 
and $f\in C^1(\R)$ satisfies that  
there exist constants $p>1$ and $C_f\geq 0$ such that 
\begin{align}\label{ass.f}
f(0)=0, \quad |f(\xi)-f(\eta)|\leq C_f(|\xi|+|\eta|)^{p-1}|\xi-\eta|, \quad \xi,\eta\in\R.
\end{align}
The damped wave equation was introduced 
by Cattaneo \cite{Cattaneo} and Vernotte \cite{Vernotte} 
to discuss an model of heat conduction with finite propagation property. 
The equation is derived by combining 
``balance law'' $u_t = {\rm div} q$ and 
``time-delayed Fourier law'' $\tau q_t + q = \nabla u$, where $q$ is the heat
flux and $\tau$ is small enough.
Therefore one can expect that the behavior of solution to \eqref{eq:ndw}
can be approximated by the one of solution to heat equation.

The aim of this study is to give
global existence of solutions to \eqref{eq:ndw} 
under the smallness of initial data in the sense of 
weighted norm
\[
\int_{\Omega}\Big(|\nabla u_0|^2+|u_0|^2+|u_1|^2\Big)\lr{x}^{2\lambda}\,dx
\]
for fixed $\lambda\geq 0$, by using the idea of 
weighted energy estimates including Kummer's confluent hypergeometric 
functions, originated in \cite{SoWa4}.

For the case $\Omega=\R^N$, there are many previous works 
dealing with the analysis of critical exponent 
in the following sense: $p_c$ is the critical exponent 
if $1<p<p_c$, then a blowup solution with sufficiently small initial data exists,
on the other hand, 
if $p>p_c$, then smallness of initial data provides the existence 
of global-in-time solutions. 
The critical exponent for compactly supported initial data 
was given by Todorova--Yordanov \cite{ToYo01} with $p_c=1+\frac{2}{N}$,
by introducing the energy estimates with exponential type weight function. 
Similar philosophy can be found in Ikehata--Tanizawa \cite{IkTa05} 
for non-compactly supported initial data.  

It should be noticed that the critical exponent 
is exactly the same as Fujita exponent for 
semilinear heat equation found in the pioneering work in Fujita \cite{Fujita66}. 
The critical case $p=p_c$ was discussed in Zhang \cite{Zhang01} 
and blowup result for small initial data is proved. 

For the framework of weak solutions in $H^1$ with non-compactly supported initial data, 
Nakao--Ono \cite{NaOn93} found that 
the existence of global-in-time solutions with sufficiently small initial data
when $1+\frac{4}{N}\leq p\leq \frac{N+2}{N-2}$. 
Later, 
Ikehata--Ohta \cite{IkOh02} discussed the critical exponent of \eqref{eq:ndw}  
with initial data $(u_0,u_1)\in (H^1\cap L^r)\times (L^2\cap L^r)$. 
It is proved that the critical exponent for this problem 
is $p_c=1+\frac{2r}{N}$: 
if $1<p<1+\frac{2r}{N}$, then nonexistence of global-in-time solutions 
occurs, and 
if $1+\frac{2r}{N}<p<\frac{N+2}{N-2}$, then global-in-time solutions 
exist for sufficiently small initial data. 

The initial data in the class 
$(H^{\alpha,0}\cap H^{0,\delta})\times (H^{\alpha-1,0}\cap H^{0,\delta})$
are discussed by Hayashi--Kaikina--Naumkin \cite{HaKaNa04DIE}, where 
$H^{\ell,m}=\{\phi\in L^2\;;\;\|\lr{x}^{m}\lr{i\pa_x}^\ell \phi\|_{L^2}<\infty\}$ 
with the Fourier multiplier $\lr{i\pa_x}=\mathcal{F}^{-1}\lr{\xi}\mathcal{F}$.
They proved the existence of global-in-time solutions (in $L^1$)
to \eqref{eq:ndw} with $p>1+\frac{2}{N}$ and 
and a heat-like asymptotic profile of solutions.
%\\
%Ikehata--Nishihara--Zhao(2006)  absorbing nonlinearity  
The analysis of \cite{HaKaNa04} 
is generalized by Ikeda--Inui--Wakasugi \cite{IkInWa17}
in the framework of $(H^{\alpha,0}\cap H^{0,\delta})\times (H^{\alpha-1,0}\cap H^{0,\delta})$ 
which can be embedded into $L^r$-space ($r\in (1,2]$).  
In their paper the critical exponent is determined as 
$1+\frac{2r}{N}$ which is the same as \cite{IkOh02}.
Recently, Inui--Ikeda--Okamoto--Wakasugi \cite{IkInOkWa_arxiv} 
discussed the critical case $p=1+\frac{2r}{N}$ under some restriction 
on $r$, which is required by a derivative loss of $L^p$-$L^q$ estimates 
for high frequency part of solutions to the linear damped wave equation. 
We note that 
for the analysis of the Cauchy problem of 
the equation $\pa_t^2u-\Delta u + \frac{\mu}{(1+t)^{\beta}}\pa_tu=f(u)$
(with time-dependent damping term), 
a similar study can be found in 
the literature (see e.g., 
Wirth \cite{Wirth04,Wirth06,Wirth07},
Nishihara \cite{Nishihara11}, 
Lin--Nishihara--Zhai \cite{LiNiZh12}, 
Wakasugi \cite{Wakasugi14}, 
Lai--Takamura--Wakasa \cite{LaTaWa17}, 
and Ikeda--Sobajima \cite{IkSo2} and their reference therein).

For the case of damped wave equation in an exterior domain, 
Ono \cite{Ono03} discussed the existence of global-in-time solutions to 
\eqref{eq:ndw} under $2\leq N\leq 6, 1+4/N+2<p<1+2/(N-2)$ 
by using the result of Dan--Shibata \cite{DaSh95}. 
On the one hand, 
Ikehata \cite{Ikehata04,Ikehata05} 
proved the existence of global-in-time solutions 
for $N=2, 2<p<\infty$ by using weighted energy estimates. 
Takeda--Ogawa \cite{TaOg09} proved non-existence 
of global-in-time solutions 
to \eqref{eq:ndw} when $N\geq 2, 1<p<1+2/N$
by employing the method of Kaplan \cite{Kaplan63} and Fujita \cite{Fujita66}. 
Note that 
in the analysis of weighted energy estimates 
of the linear problem 
with a class of space-dependent damping term
in exterior domain 
can be found in 
Ikehata \cite{Ikehata05-wee}, 
Todorova--Yordanov \cite{ToYo09},
Radu--Todorova--Yordanov \cite{RaToYo09,RaToYo10}
and 
Wakasugi--Sobajima \cite{SoWa1,SoWa2}, 
however, 
their weight function forms $e^{c(1+|x|^2)^{2-\alpha}/(1+t)^{1+\beta}}$ 
and therefore 
the initial data must have an exponential decay property.

Recently, Wakasugi--Sobajima \cite{SoWa4}
found a framework of weighted energy estimates 
with a weight function of polynomial type. 
In \cite{SoWa4}, 
the weight function is taken as 
the inverse of the positive solution of heat equation $\pa_t\Phi=\Delta\Phi$ 
including the Kummer confluent hypergeometric function (see Section 2.1 below). 
This enables us to obtain the weighted energy estimate 
of polynomial type. 

The purpose of the present paper is 
to discuss the nonlinear problem of damped wave equation 
in exterior domain in view of weighted energy estimate 
of polynomial type introduced by \cite{SoWa4}.  

To state the main result, we first give 
the definition of 
the solutions to \eqref{eq:ndw} in this paper.  
\begin{definition}[Weak solution]
\label{def:sol}
The function $u:\overline{\Omega}\times [0,T)\to \R$ is called 
a weak solution of \eqref{eq:ndw} in $(0,T)$ 
if $u$ belongs to the class
\[
S_T=
\{u\in C([0,T);H^1_0(\Omega))\cap C^1([0,T);L^2(\Omega)) \;;\;f(u(\cdot))\in C([0,T);L^2(\Omega))\}
\]
and 
$U=(u(t),\pa_tu(t))$ satisfies the  following 
integral equation in $\mathcal{H}=H^1_0(\Omega)\times L^2(\Omega)$:
\begin{align*}
U(t)=e^{-t\mathcal{A}}U_0+\int_0^te^{-(t-s)\mathcal{A}}[\mathcal{N}(U(s))]\,ds, 
\quad t\in [0,T),
\end{align*}
where $U_0=(u_0,u_1)$
$\mathcal{A}=\begin{pmatrix}0&-1\\-\Delta&1\end{pmatrix}$
with domain $D(\mathcal{A})=(H^2(\Omega)\times H^1_0(\Omega))\times H^1_0(\Omega)$
and $\mathcal{N}(u,v)=(0,f(u))$. 
\end{definition}

The existence of local-in-time solutions to 
\ref{eq:ndw} is well-known (see e.g., Ikawa \cite{Ikawa68} and Cazenave--Haraux \cite{CazenaveBOOK}). 

\begin{proposition}\label{prop:LWP}
Assume that $f$ satisfies \eqref{ass.f} with $1<p\leq \frac{N}{N-2}$. 
Then for every $(u_0,u_1)\in H^1_0(\Omega)\times L^2(\Omega)$, 
there exists a positive constant $T>0$ depending only on 
$N,p,C_0,\|u_0\|_{H^1},\|u_1\|_{L^2}$ 
such that there exists a unique weak solution $u$ in $(0,T)$. 
\end{proposition}

The notion of lifespan is the following. 
\begin{definition}[Lifespan]
\label{def:lifespan}
For the solution $u$ of \eqref{eq:ndw} with initial data 
$(u_0,u_1)\in H^1_0(\Omega)\times L^2(\Omega)$, we define lifespan 
$T_* = T_*(u_0,u_1)>0$ 
(the maximal existence time of solution $u$) is as follows:
\[
T_* = \sup \{T>0\;;\; \text{\eqref{eq:ndw} has a unique weak solution in }(0,T)\}.
\]
\end{definition}

Now we are in a position to state the main result of the present paper. 

\begin{theorem}\label{main}
Assume that $f$ satisfies \eqref{ass.f} with $1<p\leq \frac{N}{N-2}$. 
Then for every $\lambda \in [0,\frac{N}{2})$, 
there exists a positive constant $\delta_{\lambda}^*>0$ such that 
the following assertion holds:

For every $(u_0,u_1)\in H^1_0(\Omega)\times L^2(\Omega)$ satisfying
\[
\int_{\Omega}\Big(|\nabla u_0(x)|^2+|u_0(x)|^2+|u_1(x)|^2\Big)\lr{x}^{2\lambda}\,dx\leq \delta_\lambda^*, 
\] 
one has $T_*(u_0,u_1)=\infty$ when $p\geq 1+\frac{4}{N+2\lambda}$. Namely, there exists a global weak solution $u$ of \eqref{eq:ndw}
with initial data $(u_0,u_1)$.
Moreover, $u$ satisfies the following weighted estimates:
there exists a positive constant $M_{\lambda}^*>0$ such that 
\begin{gather*}
\int_{\Omega}\left(|\nabla u(x,t)|^2+|\pa_tu(x,t)|^2\right)\left(1+t+|x|^{2}\right)^{\lambda}\,dx
\leq \frac{M_{\lambda}^*}{1+t}, 
\\
\int_{\Omega}|u_1(x,t)|^2\left(1+t+|x|^{2}\right)^{\lambda}\,dx
\leq M_{\lambda}^*,
\\
\int_0^\infty\!\!
\int_{\Omega}|\nabla u(x,s)|^2\left(1+s+|x|^{2}\right)^{\lambda}\,dx\,ds
\leq M_{\lambda}^*, 
\\
\int_0^\infty(1+s)
\int_{\Omega}|\pa_t u(x,s)|^2\left(1+s+|x|^{2}\right)^{\lambda}\,dx\,ds
\leq M_{\lambda}^*.
\end{gather*}
On the other hand, for the case $1<p<1+\frac{4}{N+2\lambda}$, 
one has the following lower estimate of lifespan
\[
T_*(\ep u_0,\ep u_1)\geq C\ep^{-(\frac{1}{p-1}-\frac{N+2\lambda}{4})^{-1}}
\]
for some $C>0$ (independent of $\ep$) and sufficiently small $\ep>0$. 
\end{theorem}
\begin{remark}
In the case $\Omega=\R^N$, the global-in-time solution
of \ref{eq:ndw} 
with slowly decaying initial data (like $\lr{x}^{-\mu}$)
was constructed in Hayashi--Kaikina--Naumkin \cite{HaKaNa04DIE} 
(for $p>1+\frac{2}{N}$) under a weaker assumption than ours.
In the case of exterior domain, 
it is already discussed when $\lambda=1$. 
However, the case $\lambda\in (0,\frac{N}{2})$ 
is not dealt with so far. 
The global existence for weighted-$L^2$-type initial data 
is now established. 
\end{remark}

\begin{remark}
For $L^r$-type initial data, 
Ikeda--Inui--Okamoto--Wakasugi \cite{IkInOkWa_arxiv}
proved the case $p=1+\frac{2r}{N}$ under some restriction on $r$,
which is critical in this situation and 
related to our critical case $p=1+\frac{4}{N+2\lambda}$. 
Although their aspect is quite far form ours, 
it should be noticed that 
the framework in \cite{IkInOkWa_arxiv} 
is difficult to apply to the case of exterior domain 
because of the use of a deep Fourier analysis. 
\end{remark}

\begin{remark}
For the lifespan estimates, 
Ikeda--Inui--Wakasugi \cite{IkInWa17}
provided upper bound of lifespan estimates 
with a specific situation
\[
\Omega=\R^N, \quad
u_0+u_1 \geq \max\{1,|x|\}^{-\frac{N}{2}-\lambda}, \quad \lambda<\frac{2}{p-1}-\frac{N}{2}, \quad
f(u)=|u|^p. 
\]
They proved $T_*(\ep u_0,\ep u_1)\geq C\ep^{-(\frac{1}{p-1}-\frac{N+2\lambda}{4})^{-1}}$. 
Combining their result, we can assert that 
the lower estimate in Theorem \ref{main} is almost sharp.
\end{remark}
As a corollary of Theorem \ref{main}, we can deduce 
the existence of global-in-time solutions to 
\eqref{eq:ndw} with $p>1+\frac{2}{N}$ for polynomially decaying initial data. 

\begin{corollary}\label{cor}
Assume that $f$ satisfies \eqref{ass.f}. 
Then for every $1+\frac{2}{N}<p\leq \frac{N}{N-2}$, 
there exists a positive constant $\delta_{p}^{**}>0$ such that 
the following assertion holds:
For every $(u_0,u_1)\in H^1_0(\Omega)\times L^2(\Omega)$ satisfying
\[
\int_{\Omega}\Big(|\nabla u_0(x)|^2+|u_0(x)|^2+|u_1(x)|^2\Big)\lr{x}^{N}\,dx\leq \delta_p^{**}, 
\] 
one has $T_*(u_0,u_1)=\infty$. 
\end{corollary}

Of course we can proceed the same argument 
in the one-dimensional case $\Omega=\R$. 
However, the lack of the validity of 
(weighted) Hardy's inequality causes, and some difficulty appears. 
To avoid the use of Hardy's inequality, 
we use a solution of heat equation with some modification. 
As a result, we lose the result of the critical situation $p=1+\frac{4}{1+2\lambda}$. 
The precise statement for the case $\Omega=\R$ is written in the end of the last section. 

This paper is organized as follows.
In Section \ref{sec:pre}, 
we state the properties of a family of 
self-similar solutions to the heat equation 
including Kummer's Confluent hypergeometric  functions 
and collect some functional inequalities 
we need in the derivation of weight energy estimates.
Section \ref{sec:wee} is devoted to prove Theorem \ref{main}. 
Finally, we give a remark about the weighted energy estimates 
and global existence for one-dimensional case 
in Section \ref{sec:1d}. 

\section{Preliminaries}\label{sec:pre}
\subsection{The weight functions including confluent hypergeometric functions}
For $t_0\geq 1$ and $\beta\geq 0$, define
\[
\Phi_{\beta}(x,t:t_0)=(t_0+t)^{-\beta}
e^{-\frac{|x|^2}{4(t_0+t)}}
M\left(\frac{N}{2}-\beta,\frac{N}{2};\frac{|x|^2}{4(t_0+t)}\right), 
\quad 
(x,t)\in \R^N\times[0,\infty), 
\]
where $M(a,c;z)$ is Kummer's confluent hypergeometric function 
defined as 
\[
M(a,c,z)=\sum_{n=0}^\infty\frac{(a)_n}{(c)_n}\frac{z^n}{n!}
\]
with the Pochhammer symbol $(d)_0=1$ and $(d)_n=\prod_{k=0^{n-1}}(d+k-1)$. 
 These functions are given by Sobajima--Wakasugi \cite{SoWa4}
as a family of self-similar solution of 
linear heat equation $\pa_t\Phi-\Delta\Phi=0$.
Then we have the following lemma.
\begin{lemma}[{\cite{SoWa4}}]\label{Phi}
\begin{itemize}
\item[\bf (i)]for every $\beta\geq 0$, 
$\pa_t\Phi_\beta(x,t)=\Delta \Phi_\beta(x,t)$ 
for $(x,t)\in \R^N\times[0,\infty)$, 
\item[\bf (ii)]for every $\beta\geq 0$, 
$\pa_t\Phi_\beta(x,t)=-\beta\Phi_{\beta+1}(x,t)$ for $(x,t)\in \R^N\times[0,\infty)$,  
\item[\bf (iii)] for every $\beta\geq 0$, there exists a positive constant $C_\beta$ such that 
\[
|\Phi_\beta(x,t)|\leq C_\beta\left(t_0+t+\frac{|x|^2}{4}\right)^{-\beta},
\]
\item[\bf (iv)] for every $0\leq \beta<\frac{N}{2}$, there exists a positive constant $c_\beta$ such that 
\[
\Phi_\beta(x,t)\geq c_\beta \left(t_0+t+\frac{|x|^2}{4}\right)^{-\beta}.
\]
\end{itemize}
\end{lemma}
\subsection{Functional inequalities with weights}
In view of Lemma \ref{Phi}, for the same constant $t_0\geq 1$ as $\Phi_\beta$, 
we also introduce 
\[
\Psi(x,t:t_0)=t_0+t+\frac{|x|^2}{4}, 
\quad (x,t)\in \R^N\times [0,\infty).
\]
The following Hardy type inequality with $\Psi$ 
is also needed. 
\begin{lemma}\label{lem:hardy}
Let $\lambda>-\frac{N-2}{2}$.
For every $w\in C_c^1(\Omega)$, 
\begin{equation}\label{hardy}
4K(\lambda)^2
\int_{\Omega}
  |w|^2\Psi^{\lambda-1}
\,dx
\leq 
\int_{\Omega}
  |\nabla w|^2\Psi^{\lambda}
\,dx
\end{equation}
with $K(\lambda)=\min\{\frac{N}{2}+\lambda-1,\frac{N}{2}\}$.
That is, if $N\geq 2$, then \eqref{hardy} holds for every $\lambda>0$ 
and every $w\in H^1_0(\Omega)$ satisfying $w\Psi^{\frac{\lambda}{2}}, (\nabla w)\Psi^{\frac{\lambda}{2}}\in L^2(\Omega)$.
\end{lemma}
\begin{proof}
Noting that 
\begin{align*}
{\rm div} \left(\frac{x}{2}\Psi^{\lambda-1}\right)
=
\frac{N}{2}\Psi^{\lambda-1}+ (\lambda-1)\frac{|x|^2}{4}\Psi^{\lambda-2}
%=
%\left(\frac{N}{2}+\lambda-1\right)\frac{|x|^2}{4}\Psi^{\lambda-2}+ \frac{N}{2}(t_0+t)\Psi^{\lambda-2}
\geq 
K\Psi^{\lambda-1}
\end{align*}
with $K(\lambda)=\min\{\frac{N}{2}+\lambda-1,\frac{N}{2}\}>0$,
we see from integration by parts and 
H\"older inequality that 
\begin{align*}
K(\lambda)\int_{\R^N}|w|^2\Psi^{\lambda-1}\,dx
&
\leq 
\int_{\R^N}|w|^2{\rm div}\left(\frac{x}{2}\Psi^{\lambda-1}\right)\,dx
\\
&
\leq 
-2\int_{\R^N}w\left(\nabla w\cdot\frac{x}{2}\right)\Psi^{\lambda-1}\,dx
\\
&
\leq 
2
\left(\int_{\R^N}|w|^2\Psi^{\lambda-1}\,dx\right)^{\frac{1}{2}}
\left(\int_{\R^N}|\nabla w|^2\frac{|x|^2}{4}\Psi^{\lambda-1}\,dx\right)^{\frac{1}{2}}
\\
&
\leq 
2
\left(\int_{\R^N}|w|^2\Psi^{\lambda-1}\,dx\right)^{\frac{1}{2}}
\left(\int_{\R^N}|\nabla w|^2\Psi^{\lambda}\,dx\right)^{\frac{1}{2}}.
\end{align*}
The last assertion can be verified by the standard approximation argument 
with the molifier and cut-off functions.
\end{proof}

The following lemma is well-known 
Gagliardo--Nirenberg inequality. 

\begin{lemma}[Gagliardo--Nirenberg inequality]
\label{lem:GNineq}
If $1<p\leq \frac{N+2}{N-2}$, then there exists a constant $C_{GN,p}>0$ such that 
for every $w\in H^1_0(\Omega)$, 
\begin{align*}
\|w\|_{L^{p+1}(\Omega)}\leq 
C_{GN,p}
\|w\|_{L^2(\Omega)}^{1-\frac{N(p-1)}{2(p+1)}}
\|\nabla w\|_{L^2(\Omega)}^{\frac{N(p-1)}{2(p+1)}}.
\end{align*}
\end{lemma}

Next we give a weight version of Gagliardo-Nirenberg inequality, 
which we will exactly need in the treatment of nonlinear term $f(u)$ in \eqref{eq:ndw}. 

\begin{lemma}\label{lem:GNtype}
If $N\geq 2$ and $1<p\leq \frac{N+2}{N-2}$ and $\lambda>0$, then there exists a constant $\widetilde{C}_p>0$ such that 
for every $w\in H^1_0(\Omega)$ satisfying $v\Psi^{\lambda},(\nabla v)\Psi^\lambda\in L^2(\Omega)$, 
\begin{align*}
\|v\Psi^{\frac{\lambda}{p+1}}\|_{L^{p+1}(\Omega)}
\leq 
\widetilde{C}_p
(t_0+t)^{-\frac{\lambda(p-1)}{2(p+1)}}
\|v\Psi^{\frac{\lambda}{2}}\|_{L^2(\Omega)}^{1-\frac{N(p-1)}{2(p+1)}}
\|(\nabla v)\Psi^{\frac{\lambda}{2}}\|_{L^2(\Omega)}^{\frac{N(p-1)}{2(p+1)}}.
\end{align*}
\end{lemma}
\begin{proof}
Note that by assumption, we have
$v\Psi^{\frac{\lambda}{p+1}}\in H_0^1(\Omega)$. Therefore 
applying Lemma \ref{lem:GNineq} to $w=v\Psi^{\frac{\lambda}{p+1}}$ 
and using $|\nabla \Psi|\leq \Psi^{\frac{1}{2}}$
imply
\begin{align*}
\|v\Psi^{\frac{\lambda}{p+1}}\|_{L^{p+1}(\Omega)}
&
\leq 
C_{GN,p}
\|v\Psi^{\frac{\lambda}{p+1}}\|_{L^2(\Omega)}^{1-\frac{N(p-1)}{2(p+1)}}
\|\nabla (v\Psi^{\frac{\lambda}{p+1}})\|_{L^2(\Omega)}^{\frac{N(p-1)}{2(p+1)}}.
\\
&
\leq 
C_{GN,p}
\|v\Psi^{\frac{\lambda}{p+1}}\|_{L^2(\Omega)}^{1-\frac{N(p-1)}{2(p+1)}}
\left(
\|(\nabla v)\Psi^{\frac{\lambda}{p+1}}\|_{L^2(\Omega)}
+
\frac{\lambda}{p+1}\|v\Psi^{\frac{\lambda}{p+1}-\frac{1}{2}}\|_{L^2(\Omega)}
\right)^{\frac{N(p-1)}{2(p+1)}}.
\end{align*}
Combining the above inequality with Lemma \ref{lem:hardy} with $v$ and $\frac{\lambda}{p+1}$, we have
\begin{align*}
\|v\Psi^{\frac{\lambda}{p+1}}\|_{L^{p+1}(\Omega)}
\leq 
C_{GN,p}\left(1+\frac{\lambda K(\frac{\lambda}{p+1})}{p+1}\right)^{\frac{N(p-1)}{2(p+1)}}
\|v\Psi^{\frac{\lambda}{p+1}}\|_{L^2(\Omega)}^{1-\frac{N(p-1)}{2(p+1)}}
\|(\nabla v)\Psi^{\frac{\lambda}{p+1}}\|_{L^2(\Omega)}^{\frac{N(p-1)}{2(p+1)}}.
\end{align*}
Using the inequality $\Psi^{-1}\leq (t_0+t)^{-1}$, we deduce the 
desired inequality. 
\end{proof}

Thirdly, 
we give an inequality related to 
integration by parts formula with non-uniform weight. 
Although its proof is essentially stated in \cite{SoWa4}, 
we would give a proof for reader's convenience. 

\begin{lemma}\label{lem:ibp}
Assume that $\Phi\in C^2(\overline{\Omega})$ is positive and $\delta\in (0,\frac{1}{2})$. 
Then for every $u\in H^1_0(\Omega)$ having a compact support,
\[
\int_{\Omega}
u\Delta u
\Phi^{-1+2\delta}
\,dx
\leq 
-
\frac{\delta}{1-\delta}
\int_{\Omega}
|\nabla u|^2
\Phi^{-1+2\delta}
\,dx
+
\frac{1-2\delta}{2}
\int_{\R^N}
u^2(\Delta \Phi)
\Phi^{-2+2\delta}
\,dx
\]
\end{lemma}
\begin{proof}
Set $v=\Phi^{-1+\delta}u$. Then we have
\begin{align*}
\int_{\Omega}
|\nabla u|^2
\Phi^{-1+2\delta}
\,dx
=
\int_{\Omega}
|\nabla v|^2
\Phi
\,dx
+2(1-\delta)
\int_{\Omega}
v(\nabla v\cdot\nabla \Phi)
\,dx
+
(1-\delta)^2\int_{\Omega}
v^2\frac{|\nabla \Phi|^2}{\Phi}
\,dx.
\end{align*}
We see from integration by parts that  
\[
\int_{\Omega}
|\nabla u|^2
\Phi^{-1+2\delta}
\,dx
\geq 
-(1-\delta)
\int_{\Omega}
u^2(\Delta \Phi)\Phi^{-2+2\delta}
\,dx
+
(1-\delta)^2\int_{\Omega}u^2|\nabla \Phi|^2\Phi^{-3+2\delta}
\,dx.
\]
Using the above inequality with integration by parts twice, we deduce
\begin{align*}
\int_{\Omega}
u\Delta u
\Phi^{-1+2\delta}
\,dx
&
=
-
\int_{\Omega}
|\nabla u|^2
\Phi^{-1+2\delta}
\,dx
+
(1-2\delta)\int_{\Omega}
u(\nabla u\cdot\nabla\Phi)
\Phi^{-2+2\delta}
\,dx
\\
&
=
-
\int_{\Omega}
|\nabla u|^2
\Phi^{-1+2\delta}
\,dx
-\frac{1-2\delta}{2}\int_{\Omega}
u^2(\Delta\Phi)
\Phi^{-2+2\delta}
\,dx
\\
&\quad
+
(1-\delta)(1-2\delta)
\int_{\Omega}
u^2|\nabla\Phi|^2\Phi^{-3+2\delta}
\,dx
\\
&\leq 
\left(
\frac{1-2\delta}{1-\delta}-1\right)
\int_{\Omega}
|\nabla u|^2
\Phi^{-1+2\delta}
\,dx
+\frac{1-2\delta}{2}\int_{\Omega}
u^2(\Delta\Phi)
\Phi^{-2+2\delta}
\,dx.
\end{align*}
The proof is complete. 
\end{proof}

\section{Proof of main theorem (Theorem \ref{main})}\label{sec:wee}

Since the weak solution $u$ of \eqref{eq:ndw} can be 
approximated by the one with smooth compactly supported initial data, 
in this section 
we may assume that $u_0\in H^2(\Omega)\cap H^1_0(\Omega)$ and $u_1\in H^1_0(\Omega)$ are compactly supported 
without loss of generality. 
By finite propagation property, we also can assume that 
the solution $u(t)\in H^2(\Omega)\cap H^1_0(\Omega)$ $(t\geq 0)$ 
is also compactly supported for every $t\in [0,T_*)$. 

The proof of Theorem \ref{main} 
is based on the following proposition 
which is well-known, and so-called blowup alternative. 

\begin{proposition}\label{prop:alt}
Assume that $f$ satisfies \eqref{ass.f} with $1<p\leq \frac{N}{N-2}$. 
Let $u$ be the weak solution of \eqref{eq:ndw} in $(0,T_*)$ with the corresponding lifespan $T_*$. 
If $T_*<\infty$, then one has 
\[
\lim_{t\uparrow T_*}(\|u\|_{H^1_0(\Omega)}+\|u\|_{L^2(\Omega)})=\infty.
\]
\end{proposition}

For $t_0\geq 1$ and $\lambda>0$, define
the following weighted energy functional 
for the weak solution $u$ as follows:
\begin{equation}
\label{eq:E_lambda_1}
E_\lambda(t:t_0)
:=
(t_0+t)
\int_{\Omega}\Big(|\nabla u(x,t)|^2+|\pa_tu(x,t)|^2\Big)\Psi(x,t:t_0)^{\lambda}\,dx, 
\quad 
t\geq 0. 
\end{equation}
Then the following lemma holds. 
\begin{lemma}\label{lem:energy-1}
Let $E_\lambda$ be given in \eqref{eq:E_lambda_1}. 
Then for every $t_0\geq 1$ and $\lambda\geq 0$, 
\begin{align*}
\frac{d}{dt}E_\lambda(t:t_0)
&\leq 
(\lambda^2+\lambda+1)\int_{\Omega}|\nabla u|^2\Psi^{\lambda}\,dx
+
(\lambda+1-t_0-t)\int_{\Omega}|\pa_tu|^2\Psi^{\lambda}\,dx
\\
&\quad+
\frac{d}{dt}
\left[
2(t_0+t)\int_{\Omega}
F(u)\Psi^{\lambda}
\,dx
\right]
+
2(\lambda+1)\int_{\Omega}
F(u)\Psi^{\lambda}
\,dx, \quad\text{a.a.}\ t\in (0,T_*), 
\end{align*}
where $F(\xi)=\int_{0}^\xi f(\eta)\,d\eta$ for $\xi\in \R$.
\end{lemma}
\begin{proof}
Observe that $\pa_tu(t)\in H^1_0(\Omega)$ and $\Psi=t_0+t+\frac{|x|^2}{4}$. By integration by parts we have
\begin{align*}
\frac{d}{dt}E_\lambda(t:t_0)
&=
\int_{\Omega}\Big(|\nabla u|^2+|\pa_tu|^2\Big)[\Psi^{\lambda}+(t_0+t)\lambda\Psi^{\lambda-1}]\,dx
\\
&\quad
+
2(t_0+t)\int_{\Omega}\Big(\nabla \pa_tu\cdot\nabla u+\pa_tu\pa_t^2u\Big)\Psi^{\lambda}\,dx
\\
&\leq 
(\lambda+1)
\int_{\Omega}\Big(|\nabla u|^2+|\pa_tu|^2\Big)\Psi^{\lambda}\,dx
\\
&\quad
+
2(t_0+t)
\int_{\R^N}\pa_tu(-\Delta u+\pa_t^2u)\Psi^{\lambda}
-
\lambda \pa_tu\left(\nabla u\cdot\nabla\Psi\right)\Psi^{\lambda-1}\,dx.
\end{align*}
Since $u$ satisfies \eqref{eq:ndw}, 
the Schwarz inequality and the inequality $(t_0+t)^{\frac{1}{2}}|\nabla \Psi|\leq \Psi$ yield 
\begin{align*}
\frac{d}{dt}E_\lambda(t:t_0)
&\leq 
(\lambda+1)
\int_{\R^N}\Big(|\nabla u|^2+|\pa_tu|^2\Big)\Psi^{\lambda}\,dx
-
2
(t_0+t)
\int_{\R^N}|\pa_tu|^2\Psi^{\lambda}\,dx
\\
&\quad
+
2
(t_0+t)
\int_{\R^N}\pa_tu(f(u))\Psi^{\lambda}\,dx
\\
&\quad+
2
\lambda
(t_0+t)^{\frac{1}{2}}
\left(
\int_{\R^N}
|\pa_tu|^2\Psi^{\lambda}\,dx
\right)^{\frac{1}{2}}
\left(
\int_{\R^N}
|\nabla u|^2\Psi^{\lambda}\,dx
\right)^{\frac{1}{2}}.
\end{align*}
Noting that 
\begin{align*}
\int_{\R^N}\pa_tu(f(u))\Psi^{\lambda}\,dx
&=
\frac{d}{dt}
\left[
(t_0+t)\int_{\R^N}F(u)\Psi^{\lambda}\,dx
\right]
+
\int_{\R^N}F(u)[\Psi^{\lambda}+(t_0+t)\lambda \Psi^{\lambda-1}]\,dx
\\
&\leq 
\frac{d}{dt}
\left[
(t_0+t)\int_{\R^N}F(u)\Psi^{\lambda}\,dx
\right]
+
(\lambda+1)\int_{\R^N}F(u)\Psi^{\lambda}\,dx,
\end{align*}
%by using the inequality $2\lambda ab\leq \lambda^2a^2+b^2$ $(a,b\in\R)$ 
we deduce the desired inequality. 
\end{proof}
Next we assume $\lambda\in [0,\frac{N}{2})$. 
Set $\delta=\frac{N-2\lambda}{4N}\in (0,\frac{1}{4})$.  
Define 
\[
\beta:=\frac{\lambda}{1-2\delta}=\frac{2\lambda N}{N+2\lambda}\in \left(\lambda,\frac{N}{2}\right)
\]
and
\begin{equation}
\label{eq:E_lambda_2}
\widetilde{E}_\lambda(t:t_0)
:=
\int_{\R^N}\Big(2u(x,t)\pa_tu(x,t)+|u(x,t)|^2\Big)\Phi_\beta(x,t:t_0)^{-1+2\delta}\,dx, 
\quad 
t\in [0,T_*). 
\end{equation}
Then the following inequality holds.
\begin{lemma}\label{lem:energy-2}
Let $\widetilde{E}_{\lambda}$ be as in \eqref{eq:E_lambda_2}. 
Then for every $t_0\geq 1$ and $\lambda\in [0,\frac{N}{2})$ 
with $\delta=\frac{N-\lambda}{4N}$ and $\beta=\frac{2\lambda N}{N+2\lambda}$, one has 
\begin{align*}
\frac{d}{dt}\widetilde{E}_\lambda(t:t_0)
&\leq 
\frac{1}{c_{\beta}^{1-2\delta}}\left(
2
+\frac{(1-2\delta)\beta C_{\beta+1}}{c_\beta K(\lambda)}
\right)
\int_{\R^N}|\pa_tu|^2\Psi^{\lambda}\,dx
\\
&\quad+
\left(
\frac{(1-2\delta)\beta C_{\beta+1}}{c_\beta^{2-2\delta}t_0}
-
\frac{2\delta}{(1-\delta)C_\beta^{1-2\delta}}
\right)
\int_{\R^N}|\nabla u|^2\Psi^\lambda\,dx
+\frac{2}{c_{\beta}^{1-2\delta}}
\int_{\R^N}|uf(u)|\Psi^\lambda\,dx.
\end{align*}
\end{lemma}

\begin{proof}
Since $u$ is a solution of \eqref{eq:ndw}, 
we have 
\begin{align*}
\frac{d}{dt}\widetilde{E}_\lambda(t:t_0)
&=
2\int_{\R^N}|\pa_tu|^2\Phi_\beta^{-1+2\delta}\,dx
+2\int_{\R^N}u(\pa_t^2u+\pa_tu)\Phi_\beta^{-1+2\delta}\,dx
\\
&\quad -(1-2\delta)\int_{\R^N}\Big(2u\pa_tu+|u|^2\Big)\Phi_\beta^{-2+2\delta}\pa_t\Phi_\beta\,dx
\\
&=
2\int_{\R^N}|\pa_tu|^2\Phi_\beta^{-1+2\delta}\,dx
+2\int_{\R^N}u(\Delta u)\Phi_\beta^{-1+2\delta}\,dx
+2\int_{\R^N}uf(u)\Phi_\beta^{-1+2\delta}\,dx
\\
&
\quad 
-2(1-2\delta)\int_{\R^N}u\pa_tu\Phi_\beta^{-2+2\delta}\pa_t\Phi_\beta\,dx
-(1-2\delta)\int_{\R^N}|u|^2\Phi_\beta^{-2+2\delta}\pa_t\Phi_\beta\,dx.
\end{align*}
Employing Lemma \ref{lem:ibp} with $\Phi=\Phi_\beta$,  we have
\begin{align*}
\frac{d}{dt}\widetilde{E}_\lambda(t:t_0)
&\leq
2\int_{\R^N}|\pa_tu|^2\Phi_\beta^{-1+2\delta}\,dx
-\frac{2\delta}{1-\delta}\int_{\R^N}|\nabla u|^2\Phi_\beta^{-1+2\delta}\,dx
+2\int_{\R^N}uf(u)\Phi_\beta^{-1+2\delta}\,dx
\\
&
\quad 
-2(1-2\delta)\int_{\R^N}u\pa_tu\Phi_\beta^{-2+2\delta}\pa_t\Phi_\beta\,dx
-(1-2\delta)\int_{\R^N}|u|^2\Phi_\beta^{-2+2\delta}(\pa_t\Phi_\beta-\Delta\Phi_\beta)\,dx.
\end{align*}
Here we use the profile of $\Phi_\beta$ stated in Lemma \ref{Phi}. 
Then 
\begin{align*}
\frac{d}{dt}\widetilde{E}_\lambda(t:t_0)
&\leq
\frac{2}{c_\beta^{1-2\beta}}\int_{\R^N}|\pa_tu|^2\Psi^{\lambda}\,dx
-
\frac{2\delta}{(1-\delta)C_\beta^{1-2\delta}}\int_{\R^N}|\nabla u|^2\Psi^\lambda\,dx
+\frac{2}{c_\beta^{1-2\beta}}\int_{\R^N}|uf(u)|\Psi^\lambda\,dx
\\
&
\quad 
+\frac{2(1-2\delta)\beta C_{\beta+1}}{c_\beta^{2-2\delta}}
\int_{\Omega}|u|\,|\pa_tu|\Psi^{\lambda-1}\,dx.
\end{align*}
By Lemma \ref{lem:hardy}, 
the last term on the right-hand side of the above inequality 
can be estimated as follows:
\begin{align*}
\int_{\Omega}|u|\,|\pa_tu|\Psi^{\lambda-1}\,dx
&\leq 
\left(\int_{\Omega}|u|^2\Psi^{\lambda-1}\,dx\right)^{\frac{1}{2}}
\left(\int_{\Omega}|\pa_tu|^2\Psi^{\lambda-1}\,dx\right)^{\frac{1}{2}}
\\
&\leq 
\frac{1}{K(\lambda)(t_0+t)^{\frac{1}{2}}}
\left(\int_{\Omega}|\nabla u|^2\Psi^{\lambda}\,dx\right)^{\frac{1}{2}}
\left(\int_{\Omega}|\pa_tu|^2\Psi^{\lambda}\,dx\right)^{\frac{1}{2}}.
\end{align*}
%Noting that $K^{-1}(t_0+t)^{-\frac{1}{2}} ab\leq K^{-2}a^2+(t_0+t)^{-1}b^2$ $(a,b\in\R)$, 
Hence we obtain 
the desired inequality. 
\end{proof}

To the end of this section 
we will give an estimate for the following weighted total energy functional:
\begin{align}
\nonumber 
m_\lambda(t:t_0)
&:=
(t_0+t)\int_{\R^N}\Big(|\nabla u(x,t)|^2+|\pa_tu(x,t)|^2\Big)\Psi(x,t:t_0)^{\lambda}\,dx
\\
\label{eq:total}
&\quad+\int_{\R^N}|u(x,t)|^2\Psi(x,t:t_0)^{\lambda}\,dx, 
\quad 
t\in [0,T_*). 
\end{align}
\begin{lemma}\label{lem:equiv}
for every $\nu>0$ there exists positive constants $\gamma_\nu>0$ and $\Gamma_\nu>0$ 
such that 
if $t_0\geq c_\beta^{-1+2\delta}\nu$, then  
\[
\gamma_\nu 
m_\lambda(t:t_0)
\leq 
E_\lambda(t:t_0)
+
\nu 
\widetilde{E}_\lambda(t:t_0)
\leq 
\gamma_\nu 
m_\lambda(t:t_0), 
\quad t\in (0,T_*).
\]
\end{lemma}
\begin{proof}
By the Schwarz inequality and Lemma \ref{Phi} {\bf (iv)} 
we see that 
\begin{align*}
\left|
\int_{\Omega}
  u\pa_tu\Phi_\beta^{-1+2\delta}
\,dx
\right|
&\leq 
\left(
\int_{\Omega}
  |\pa_tu|^2\Phi_\beta^{-1+2\delta}
\,dx
\right)^{\frac{1}{2}}
\left(
\int_{\Omega}
  |u|^2\Phi_\beta^{-1+2\delta}
\,dx
\right)^{\frac{1}{2}}
\\
&\leq 
\frac{1}{c_\beta^{\frac{1}{2}-\delta}(t_0+t)^{\frac{1}{2}}}
E_\lambda(t:t_0)^{\frac{1}{2}}
\left(
\int_{\Omega}
  |u|^2\Phi_\beta^{-1+2\delta}
\,dx
\right)^{\frac{1}{2}}
\\
&\leq 
\frac{1}{c_\beta^{1-2\delta}t_0}
E_\lambda(t:t_0)
+
\frac{1}{4}
\int_{\Omega}
  |u|^2\Phi_\beta^{-1+2\delta}
\,dx
\end{align*}
and therefore 
\begin{align*}
E_\lambda(t:t_0)
+\nu \widetilde{E}_\lambda(t:t_0)
&\geq
\left(1-
\frac{\nu}{c_\beta^{1-2\delta}t_0}\right)
E_\lambda(t:t_0)
+
\frac{\nu}{2}
\int_{\Omega}
  |u|^2\Phi_\beta^{-1+2\delta}
\,dx
\\
E_\lambda(t:t_0)
+\nu \widetilde{E}_\lambda(t:t_0)
&\leq
\left(1+
\frac{\nu}{c_\beta^{1-2\delta}t_0}\right)
E_\lambda(t:t_0)
+
\frac{3\nu}{2}
\int_{\Omega}
  |u|^2\Phi_\beta^{-1+2\delta}
\,dx.
\end{align*}
In view of Lemma \ref{Phi} {\bf (iii)}, 
this means that the assumption $t_0\geq c_\beta^{-1+2\delta}\nu$ 
implies the assertion of this lemma. 
\end{proof}
Furthermore we set 
\begin{align*}
Y_\lambda(t:t_0)
&:=
\int_0^t\!\!\int_{\Omega}|\nabla u(x,s)|^2\Psi(s:t_0)^\lambda\,dx\,ds, \quad t\in [0,T_*)
\\
Z_\lambda(t:t_0)
&:=
\int_0^t(t_0+s)\int_{\Omega}|\pa_tu(x,s)|^2\Psi(s:t_0)^\lambda\,dx\,ds, 
\quad t\in [0,T_*).
\end{align*}
\begin{proposition}\label{prop:WEE}
There exists positive constants  $t_0^*\geq 1$ and $\eta^*>0$ 
such that 
\begin{align*}
&\eta^*
\left(
m_\lambda(t:t_0^*)
+
Y_\lambda(t:t_0^*)
+
Z_\lambda(t:t_0^*)
\right)
\\
&\leq
m_\lambda(0:t_0^*)
+\int_{\R^N}|F(u_0)|\Psi_*(0)^\lambda\,dx
\\
&\quad
+(t_0^*+t)\int_{\R^N}|F(u(t))|\Psi_*(t)^\lambda\,dx
+
\int_0^t\!\!
\int_{\R^N}
(|F(u(s))|+|u(s)f(u(s))|)\Psi_*(s)^{\lambda}
\,dx\,ds, 
\end{align*}
where $\Psi_*(\cdot,t)=\Psi(\cdot,t:t_0^*)$. 
\end{proposition}
\begin{remark}
If $f\equiv 0$, then \eqref{eq:ndw} is linear problem of 
damped wave equation in exterior domain.  
In this case Proposition \ref{prop:WEE} provides the following energy decay estimates
\begin{gather*}
\int_{\R^N}|\nabla u(t)|^2\,dx
+\int_{\R^N}|\pa_t u(t)|^2\,dx
\leq 
C(1+t)^{-\lambda-1},
\\
\int_{\R^N}|u(t)|^2\,dx
\leq 
C(1+t)^{-\lambda}
\end{gather*}
under the assumption $(|\nabla u_0|^2+|u_0|^2+|u_1|^2)\lr{x}^{\lambda}\in L^1(\Omega)$ with $\lambda\in [0,\frac{N}{2})$, 
which is slightly weaker than that of \cite{SoWa4}. 
\end{remark}
\begin{proof}
We see from Lemmas \ref{lem:energy-1} and \ref{lem:energy-2} that 
if there exists a constant $t_1>0$ such that if 
$t_0\geq t_1$, then we have   
\begin{align*}
\frac{d}{dt}
E_\lambda(t:t_0)
&\leq 
(\lambda^2+\lambda+1)\int_{\Omega}|\nabla u|^2\Psi^{\lambda}\,dx
-\frac{1}{2}(t_0+t)\int_{\Omega}|\pa_tu|^2\Psi^{\lambda}\,dx
\\
&\quad
+
\frac{d}{dt}
\left[
2(t_0+t)\int_{\Omega}
F(u)\Psi^{\lambda}
\,dx
\right]
+
2(\lambda+1)\int_{\Omega}
F(u)\Psi^{\lambda}
\,dx
\\
c_{\beta}^{1-2\delta}
\frac{d}{dt}
\widetilde{E}_\lambda(t:t_0)
&\leq 
\left(
2
+\frac{(1-2\delta)\beta C_{\beta+1}}{c_\beta K(\lambda)}
\right)
\int_{\R^N}|\pa_tu|^2\Psi^{\lambda}\,dx
-
\frac{\delta c_\beta^{1-2\delta}}{(1-\delta)C_\beta^{1-2\delta}}
\int_{\R^N}|\nabla u|^2\Psi^\lambda\,dx
\\
&\quad
+2
\int_{\R^N}|uf(u)|\Psi^\lambda\,dx.
\end{align*}
Therefore by choosing $\nu=\delta^{-1}(1-\delta)C_\beta^{1-2\delta}(\lambda^2+\lambda+2)$
and $t_0^*$ sufficiently large, 
we have
\begin{align*}
\frac{d}{dt}\Big[
E_\lambda(t:t_0)
+\nu \widetilde{E}_\lambda(t:t_0)
\Big]
&\leq 
\left[
\nu 
\left(
2
+\frac{(1-2\delta)\beta C_{\beta+1}}{c_\beta K(\lambda)}
\right)
-\frac{1}{2}(t_0+t)
\right]
\int_{\Omega}|\pa_tu|^2\Psi^{\lambda}\,dx
\\
&\quad
-\int_{\Omega}|\nabla u|^2\Psi^{\lambda}\,dx
+
\frac{d}{dt}
\left[
2(t_0+t)\int_{\Omega}
F(u)\Psi^{\lambda}
\,dx
\right]
\\
&\quad
+
2(\lambda+1)\int_{\Omega}
F(u)\Psi^{\lambda}
\,dx
+
2\nu 
\int_{\R^N}|uf(u)|\Psi^\lambda\,dx
\\
&\leq 
-\frac{1}{4}(t_0+t)
\int_{\Omega}|\pa_tu|^2\Psi^{\lambda}\,dx
-\int_{\Omega}|\nabla u|^2\Psi^{\lambda}\,dx
\\
&\quad
+
\frac{d}{dt}
\left[
2(t_0+t)\int_{\Omega}
F(u)\Psi^{\lambda}
\,dx
\right]
\\
&\quad
+
2(\lambda+1)\int_{\Omega}
F(u)\Psi^{\lambda}
\,dx
+
2\nu 
\int_{\R^N}|uf(u)|\Psi^\lambda\,dx
\end{align*}
Integrating it over $[0,t]$ and applying Lemma \ref{lem:equiv}, we have
\begin{align*}
&
\gamma_\nu m_\lambda(t:t_0)
+
\int_0^t\!\!\int_{\Omega}|\nabla u(s)|^2\Psi(s)^\lambda\,dx\,ds
+
\frac{1}{4}
\int_0^t(t_0^*+s)\int_{\Omega}|\pa_tu(s)|^2\Psi(s)^\lambda\,dx\,ds
\\
&\leq 
\Gamma_\nu m_\lambda(0:t_0)
-
2t_0\int_{\Omega}
F(u_0)\Psi(0)^{\lambda}
\,dx
\\
&\quad
+
2(t_0+t)\int_{\Omega}
F(u(t))\Psi(t)^{\lambda}
\,dx
+
\int_0^t\!\!\int_{\Omega}
\Big(2(\lambda+1)F(u(s))+2\nu |u(s)f(u(s))|)\Big)\Psi(s)^{\lambda}
\,dx\,ds,
\end{align*}
where $\Psi_*(t)=\Psi(\cdot,t:t_0^*)$. 
This yields the desired inequality. 
\end{proof}

\begin{proof}[Proof of Theorem \ref{main}]
Put $m_{\lambda}(t)=m_\lambda(t:t_0^*)$ and $Y_{\lambda}(t)=Y_\lambda(t:t_0^*)$.
By Proposition \ref{prop:WEE} 
and \eqref{ass.f}, we deduce
\begin{align}
\nonumber 
\eta^*
\Big(m_\lambda(t)+Y_\lambda(t)\Big)
&
\leq
m_\lambda^*+\frac{C_f}{p+1}(t_0^*+t)\int_{\R^N}|u(t)|^{p+1}\Psi_*(t)^\lambda\,dx
\\
\label{eq:proofmain}
&\quad+
\frac{(p+2)C_f}{p+1}
\int_0^t\!\!
\int_{\Omega}
|u(s)|^{p+1}\Psi_*(s)^{\lambda}
\,dx\,ds, 
\end{align}
where 
\[
m_{\lambda}^*:=
m_\lambda(0)+\frac{C_f}{p+1}\int_{\R^N}|u_0|^{p+1}\Psi_*(0)^\lambda\,dx<\infty.
\]
{\bf (The supercritical case $1+\frac{4}{N+2\lambda}<p\leq \frac{N}{N-2}$)} Observe that Lemma \ref{lem:GNtype} that 
\begin{align*}
&\int_{\Omega}|u(t)|^{p+1}\Psi_*(t)^\lambda\,dx
\\
&\leq 
\widetilde{C}_p^{p+1}(t_0^*+t)^{-\frac{\lambda}{2}(p-1)}
\left(
\int_{\Omega}|u(t)|^{2}\Psi_*(t)^\lambda\,dx
\right)^{\frac{p+1}{2}-\frac{N(p-1)}{4}}
\left(
\int_{\Omega}|\nabla u(t)|^{2}\Psi_*(t)^\lambda\,dx
\right)^{\frac{N(p-1)}{4}}
\\
&\leq 
\widetilde{C}_p^{p+1}(t_0^*+t)^{-\frac{N+2\lambda}{4}(p-1)}
m_{\lambda}(t)^{\frac{p+1}{2}}. 
\end{align*}
Therefore from \eqref{eq:proofmain}
we obtain the following integral inequality for $m_{\lambda}(t)$:
\begin{align*}
\eta^*
m_\lambda(t)
&\leq
m_\lambda^*
+
\frac{C_f\widetilde{C}_p^{p+1}}{p+1}
(t_0^*+t)^{1-\frac{(N+2\lambda)(p-1)}{4}}m_\lambda(t)^{\frac{p+1}{2}}
\\
&\quad+
\frac{(p+2)C_f\widetilde{C}_p^{p+1}}{p+1}
\int_0^t\!\!
(t_0^*+t)^{-\frac{(N+2\lambda)(p-1)}{4}}m_\lambda(s)^{\frac{p+1}{2}}
\,ds, \quad t\in [0,T_*). 
\end{align*}
Consequently, setting 
\[
M_\lambda(t):=\sup_{0\leq \tau\leq t}m_\lambda(\tau), \quad t\in [0,T_*)
\] 
we see from the assumption $p>1+\frac{4}{N+2\lambda}$ that $\frac{(N+2\lambda)(p-1)}{4}>1$  
\[
\eta^*	M_{\lambda}(t)
\leq 
m_{\lambda}^*
+
\frac{C_f\widetilde{C}_p^{p+1}}{p+1}
\left(
1
+
\frac{p+2}{1-\frac{(N+2\lambda)(p-1)}{4}}
\right)M_\lambda(t)^{\frac{p+1}{2}}.
\]
It is worth noticing that $M_\lambda$ is continuous.  
This implies that there exist constants $\delta_\lambda\*>0$ 
and $M_\lambda^*>0$ such that 
if $m_{\lambda}^*\leq \delta_\lambda^*$, then $M_\lambda(t)\leq M_\lambda^*$ for every $t\in [0,T_*)$, that is, we obtain 
\[
(t_0+t)
\int_{\R^N}\Big(|\nabla u(x,t)|^2+|\pa_tu(x,t)|^2\Big)\Psi(x,t)^{\lambda}\,dx
+
\int_{\R^N}|u(x,t)|^2\Psi(x,t)^{\lambda}\,dx
\leq M_\lambda^*, \quad t\in(0,T_*). 
\]
{\bf (The critical case $p=1+\frac{4}{N+2\lambda}$)} In this case, $Y_\lambda$ plays an important role. 
Note that H\"older inequality yields
\begin{align*}
\int_{\Omega}|u|^{p+1}\Psi^{\lambda}\,dx
\leq 
\left(
\int_{\Omega}|u|^2\Psi^{\lambda-1}\,dx
\right)^{1-\theta}
\left(
\int_{\Omega}|u|^{q+1}\Psi^{\frac{q+1}{2}\lambda}\,dx
\right)^{\theta}
\end{align*}
with $\theta=\frac{N}{N+2\lambda}\in(0,1)$ and 
$q=1+\frac{4}{N}$.  By Lemma \ref{lem:GNineq} 
with $p=q$ and Lemma \ref{lem:hardy}, we deduce
\[
\int_{\Omega}|u|^{q+1}\Psi^{\frac{q+1}{2}\lambda}\,dx
\leq 
C_*
%\left(\int_{\Omega}|\nabla u|^2\Psi^{\lambda}\,dx\right)^{\frac{N}{4}(q-1)}
%\left(\int_{\Omega}|u|^{2}\Psi^{\lambda}\,dx\right)^{\frac{q+1}{2}-\frac{N}{4}(q-1)}.
\int_{\Omega}|\nabla u|^2\Psi^{\lambda}\,dx
\left(\int_{\Omega}|u|^{2}\Psi^{\lambda}\,dx\right)^{\frac{2}{N}}.
\]
Combining the above two estimates and using Lemma \ref{lem:hardy} again, we have 
\begin{align*}
\int_{\Omega}|u|^{p+1}\Psi^{\lambda}\,dx
&\leq 
C_*^\theta
%\left(
%\int_{\Omega}|\nabla u|^2\Psi^{\lambda}\,dx
%\right)^{1-\theta+\frac{N}{4}(q-1)\theta}
%\left(\int_{\Omega}|u|^{2}\Psi^{\lambda}\,dx\right)^{[\frac{q+1}{2}-\frac{N}{4}(q-1)]\theta}
\left(
\int_{\Omega}|\nabla u|^2\Psi^{\lambda}\,dx
\right)^{1-\theta+\theta}
\left(\int_{\Omega}|u|^{2}\Psi^{\lambda}\,dx\right)^{\frac{2\theta}{N}}
\\
&=
C_*^\theta
\int_{\Omega}|\nabla u|^2\Psi^{\lambda}\,dx
\left(\int_{\Omega}|u|^{2}\Psi^{\lambda}\,dx\right)^{\frac{p-1}{2}}.
\end{align*}
Therefore we see from \ref{eq:proofmain} that 
\begin{align}
\nonumber 
\eta^*
\Big(m_\lambda(t)+Y_\lambda(t)\Big)
&
\leq
m_\lambda^*+\frac{C_*^\theta C_f}{p+1}m_\lambda(t)^{\frac{p+1}{2}}
\\
\label{eq:proofmain2}
&\quad+
\frac{(p+2)C_*^\theta C_f}{p+1}
\int_0^t m_{\lambda}(s)^{\frac{p-1}{2}}
\int_{\Omega}
|\nabla u(s)|^{2}\Psi_*(s)^{\lambda}
\,dx\,ds, 
\end{align}
Choosing 
\[
\widetilde{M}_{\lambda}(t):=M_{\lambda}(t)+Y_{\lambda}(t), \quad t\in [0,T_*)
\]
and noting that 
\[
\int_0^t m_{\lambda}(s)^{\frac{p-1}{2}}
\int_{\Omega}
|\nabla u(s)|^{2}\Psi_*(s)^{\lambda}
\,dx\,ds
\leq 
M_{\lambda}(t)^{\frac{p-1}{2}}
Y_{\lambda}(t)
\leq 
\widetilde{M}_{\lambda}(t)^{\frac{p+1}{2}},
\]
we have
\[
\widetilde{M}_{\lambda}(t)
\leq m_{\lambda}^*
+
\frac{(p+3)C_*^\theta C_f}{p+1}
\widetilde{M}_{\lambda}(t)^{\frac{p+1}{2}}. 
\]
The rest of the proof is exactly the same as the supercritical case. 
The proof is complete. 
\end{proof}
\begin{remark}
Similar argument as the critical case also works 
when 
$1+\frac{4}{N+2\lambda}<p<\{\frac{N}{N-2}, 1+\frac{2}{\lambda}\}$. 
\end{remark}

\section{Remark on one-dimensional case}\label{sec:1d}
In this section we consider the one-dimensional case
\begin{equation}\label{eq:ndw_1d}
\begin{cases}
\pa_t^2u(x,t)-\pa_x^2 u(x,t) + \pa_tu(x,t) = f\big(u(x,t)\big), 
& (x,t) \in \R\times (0,T),
\\
u(x,0)=u_0(x), 
& x \in \R,
\\
\pa_tu(x,0)=u_1(x),
& x \in \R.
\end{cases}
\end{equation}
In this case, we can also discuss 
the weighted energy estimate of polynomial type for $\lambda\in [0,\frac{1}{2})$. 
However, the lack of validity of Hardy type inequality (Lemma \ref{lem:hardy}), 
we take a small modification of the weight function in \eqref{eq:E_lambda_2} as follows: 
\begin{align*}
\widetilde{\Phi}_{\beta}(x,t:t_0)=\left(2-\frac{1}{t_0+t}\right)\Phi_\beta(x,t:t_0)
\end{align*}
with $\delta=\frac{1-2\lambda}{4}$ and $\beta=\frac{2\lambda}{1+2\lambda}$.
Then by virtue of the properties of $\Phi_\beta$ in Lemma \ref{Phi}, 
we have 
\[
\pa_t\widetilde{\Phi}_{\beta}\geq \Delta\widetilde{\Phi}_{\beta}+\frac{1}{(1+t)^2}\widetilde{\Phi}_{\beta}
\]
and hence
we can proceed an argument similar to the one in Section \ref{sec:wee} with $p>1+\frac{4}{N+2\lambda}=1+\frac{4}{1+2\lambda}$. 
It should be noticed that 
the case $p=1+\frac{4}{1+2\lambda}$ cannot be treated because of 
the lack of the validity of weighted Hardy inequality. 
Therefore we have the following estimate
\[
\eta^* M_{\lambda}(t)\leq 
m_{\lambda}^*+\widetilde{C}_\ep M_{\lambda}(t)^{\frac{p+1}{2}}\quad 
t\in [0,T_*).
\]
Consequently, we can obtain
\begin{theorem}\label{main_1d}
Assume that $N=1$ and $f$ satisfies \eqref{ass.f} with $1<p<\infty$. 
Then for every $\lambda \in [0,\frac{1}{2})$, 
there exists a positive constant $\delta_{\lambda}^*>0$ such that 
the following assertion holds:
For every $(u_0,u_1)\in H^2(\R)\times H^1(\R)$ satisfying
\[
\int_{\R}\Big(|\nabla u_0|^2+|u_0|^2+|u_1|^2\Big)\lr{x}^{2\lambda}\,dx\leq \delta_\lambda^*, 
\] 
one has $T_*(u_0,u_1)=\infty$ when $p>1+\frac{4}{1+2\lambda}$. Namely, there exists a global weak solution $u$ of \eqref{eq:ndw_1d}
with initial data $(u_0,u_1)$.
Moreover, $u$ satisfies the following weighted energy estimate:
there exists a positive constant $M_{\lambda}^*>0$ such that 
\begin{gather*}
\int_{\R}\left(|\nabla u(t)|^2+|\pa_tu(t)|^2\right)\left(1+t+|x|^{2}\right)^{\lambda}\,dx
\leq \frac{M_{\lambda}^*}{1+t}, 
\\
\int_{\R}|u_1(t)|^2\left(1+t+|x|^{2}\right)^{\lambda}\,dx
\leq M_{\lambda}^*.
\end{gather*}
\end{theorem}

\subsection*{Acknowedgements}
This work is partially supported 
by Grant-in-Aid for Young Scientists Research 
No.18K134450.
The author also thanks Prof.\ Yuta Wakasugi 
for giving a valuable comment for the result, 
which helps the author to accomplish to close the paper.

%%%%%%%%%%%%%%                   %%%%%%%%%%%%%%
%%%%%%%%%%%                         %%%%%%%%%%%
%%%%%%%%           References          %%%%%%%%
%%%%%%%%%%%                         %%%%%%%%%%%
%%%%%%%%%%%%%%                   %%%%%%%%%%%%%%
%A B C D E F G H I J K L M N O P Q R S T U V W X Y Z


\begin{thebibliography}{30}
\bibitem{Cattaneo}
 C. Cattaneo, {\it Sur une forme de l’\'equation de la chaleur \'eliminant le paradoxe d'une propagation
instantan\'ee}, C.\ R.\ Acad.\ Sci.\ {\bf 247} (1958), 431--433.
\bibitem{CazenaveBOOK}
	T. Cazenave, A. Haraux, 
	{\it An Introduction to Semilinear Evolution Equation}, Oxford Science, 1998.
\bibitem{DaSh95}
	W. Dan, Y. Shibata, 
	{\it On a local energy decay of solutions of a dissipative wave equation},
	Funkcial.\ Ekvac.\ {\bf 38} (1995), 545--568. 
\bibitem{Fujita66}
H. Fujita, 
{\it On the blowing up of solutions of the Cauchy problem for $u_t=\Delta u+u^{1+\alpha}$}, 
J.\ Fac.\ Sci.\ Univ.\ Tokyo Sect.\ I {\bf 13} (1966), 109--124.
\bibitem{HaKaNa04DIE}
N. Hayashi, E.I. Kaikina, P.I. Naumkin, 
{\it Damped wave equation with super critical nonlinearities}. Differential Integral Equations {\bf 17} (2004), no. 5-6, 637--652. 
\bibitem{Ikawa68}
M. Ikawa, {\it Mixed problem for hyperbolic equations of second order}, J.\ Math.\ Soc.\ Japan {\bf 20} (1968), 580--608.
\bibitem{IkInWa17}
M. Ikeda, T. Inui, Y. Wakasugi, 
{\it The Cauchy problem for the nonlinear damped wave equation with slowly decaying data}, 
NoDEA Nonlinear Differential Equations Appl.\ {\bf 24} (2017), Art. 10, 53 pp.
\bibitem{IkInOkWa_arxiv}
	M. Ikeda, T. Inui, M. Okamoto, Y. Wakasugi, 
	{\it $L^p$-$L^q$ estimates for the damped wave equation and the critical exponent for the nonlinear problem with slowly decaying data}, 
	arXiv:1710.06538.
\bibitem{IkSo2}
	M. Ikeda, M. Sobajima, 
	{\it Life-span of solutions to semilinear wave equation with time-dependent critical damping for specially localized initial data}, 
	Math.\ Ann., to appear(doi:{\tt 10.1007/s00208-018-1664-1}).
\bibitem{Ikehata04}
	R. Ikehata, 
	{\it Global existence of solutions for semilinear damped wave equation in $2$-D exterior domain}, 
	J.\ Differential Equations {\bf 200} (2004), 53--68. 
 \bibitem{Ikehata05}
	R. Ikehata, 
	{\it Two dimensional exterior mixed problem for semilinear damped wave equations}, 
	J.\ Math.\ Anal.\ Appl.\ {\bf 301} (2005), 366--377.
\bibitem{Ikehata05-wee}
	R. Ikehata, 
	{\it Some remarks on the wave equation with potential type damping coefficients}, 
	Int.\ J.\ Pure Appl.\ Math.\ {\bf 21} (2005), 19--24.
\bibitem{IkOh02}
	R. Ikehata, M. Ohta, 
	{\it Critical exponents for semilinear dissipative wave equations in $\R^N$}, 
	J.\ Math.\ Anal.\ Appl.\ {\bf 269} (2002), 87--97.
\bibitem{IkTa05}
	R. Ikehata, K. Tanizawa, 
	{\it Global existence of solutions for semilinear damped wave equations in $R^N$ with noncompactly supported initial data},
	Nonlinear Anal.\ {\bf 61} (2005), 1189--1208. 
 \bibitem{Kaplan63}
	S. Kaplan, 
{\it On the growth of solutions of quasi-linear parabolic equations}, 
Comm.\ Pure Appl.\ Math.\ {\bf 16} (1963), 305--330. 
\bibitem{LaTaWa17}
	N.-A. Lai, H. Takamura, K. Wakasa, 
{\it Blow-up for semilinear wave equations with the scale invariant damping and super-Fujita exponent}, 
J.\ Differential Equations {\bf 263} (2017), 5377--5394.
\bibitem{LiNiZh12}
	J. Lin, K. Nishihara, J. Zhai, 
	{\it Critical exponent for the semilinear wave equation with time-dependent damping}, 
	Discrete Contin.\ Dyn.\ Syst.\ {\bf 32} (2012), 4307--4320.
\bibitem{NaOn93}
	M. Nakao, K. Ono, 
	{\it Existence of global solutions to the Cauchy problem for the semilinear dissipative wave equations},
	Math.\ Z.\ {\bf 214} (1993), 325--342. 
\bibitem{Nishihara11}
	K. Nishihara, 
	{\it Asymptotic behavior of solutions to the semilinear wave equation with time-dependent damping}, 
	Tokyo J.\ Math.\ {\bf 34} (2011), 327--343. 
\bibitem{Ono03}
	K. Ono, 
	{\it Decay estimates for dissipative wave equations in exterior domains},
	J.\ Math.\ Anal.\ Appl.\ {\bf 286} (2003), 540--562. 
\bibitem{RaToYo09}
	P. Radu, G. Todorova, B. Yordanov, 
	{\it Higher order energy decay rates for damped wave equations with variable coefficients}, 
	Discrete Contin.\ Dyn.\ Syst.\ Ser.\ S {\bf 2} (2009), 609--629. 
\bibitem{RaToYo10}
	P. Radu, G. Todorova, B. Yordanov, 
	{\it Decay estimates for wave equations with variable coefficients}, 
	Trans.\ Amer.\ Math.\ Soc.\ {\bf 362} (2010) 2279--2299. 
 \bibitem{SoWa1}
	M. Sobajima, Y. Wakasugi, 
	{\it Diffusion phenomena for the wave equation with space-dependent damping in an exterior domain}, 
	J.\ Differential Equations {\bf 261} (2016), 5690--5718.
\bibitem{SoWa2}
	M. Sobajima, Y. Wakasugi, 
	{\it Remarks on an elliptic problem arising in weighted energy estimates for wave equations with space-dependent damping term in an exterior domain}, 
 AIMS Mathematics {\bf 2} (2017), 1--15.  
\bibitem{SoWa4}
	M. Sobajima, Y. Wakasugi, 
	{\it Weighted energy estimates for wave equation with space-dependent damping term 
	for slowly decaying initial data},  
    Communications in Contemporary Mathematics, to appear. 
\bibitem{TaOg09}
	T. Ogawa, H. Takeda, 
	{\it Non-existence of weak solutions to nonlinear damped wave equations in exterior domains}, 
	Nonlinear Anal.\ {\bf 70} (2009), 3696--3701. 
\bibitem{ToYo01}
G. Todorova, B. Yordanov, 
{\it Critical exponent for a nonlinear wave equation with damping}, 
J.\ Differential Equations {\bf 174} (2001), 464--489.
\bibitem{ToYo09}
	G. Todorova, B. Yordanov, 
	{\it Weighted $L^2$-estimates of dissipative wave equations with variable coefficients}, 
	J.\ Differential Equations {\bf 246} (2009), 4497--4518. 
\bibitem{Vernotte}
    P. Vernotte, 
    {\it Les paradoxes de la th\'eorie continue de l'\'equation de la chaleur},
    Comptes Rendus {\bf 246} (1958), 3154--3155.
\bibitem{Wakasugi14}
    Y. Wakasugi, 
    {\it Critical exponent for the semilinear wave equation with scale invariant damping}, 
    Fourier analysis, 375--390, Trends Math., Birkhauser/Springer, Cham, 2014.
\bibitem{Wirth04}
	J. Wirth, 
	{\it Solution representations for a wave equation with weak dissipation}, 
	Math.\ Methods Appl.\ Sci.\ {\bf 27} (2004), 101--124.
\bibitem{Wirth06}
	J. Wirth, 
	{\it Wave equations with time-dependent dissipation. I. Non-effective dissipation}, 
	J.\ Differential Equations {\bf 222} (2006), 487--514. 
\bibitem{Wirth07}
	J. Wirth, 
	{\it Wave equations with time-dependent dissipation. II. Effective dissipation}, 
	J.\ Differential Equations {\bf 232} (2007), 74--103. 
 \bibitem{Zhang01}
Q.S. Zhang, 
{\it A blow-up result for a nonlinear wave equation with damping: the critical case}, 
C.\ R.\ Acad.\ Sci.\ Paris S\'er.\ I Math.\ {\bf 333} (2001), 109--114.
\end{thebibliography}
\end{document}